\documentclass[11pt,a4paper]{amsart}
\usepackage{hyperref}
\usepackage[active]{srcltx}
\usepackage[T1]{fontenc}
\usepackage[latin1]{inputenc}
\usepackage{amssymb,amsthm,amsfonts}
\usepackage[mathcal]{euscript}
\usepackage{stmaryrd}
\usepackage{enumitem}
\usepackage{xspace}
\usepackage{microtype}
\usepackage{xcolor}

\newtheorem{Thm}{Theorem}
\newtheorem{Prop}[Thm]{Proposition}

\newtheorem{Cor}[Thm]{Corollary}

\newtheorem*{definition*}{Definition}

\numberwithin{equation}{section}

\def\pv#1{\ensuremath{{\sf#1}}}
\def\Cl#1{\ensuremath{\mathcal #1}}
\def\Om#1#2{\ensuremath{\overline\Omega_{#1}{\pv{#2}}}}
\def\tilom#1{\ensuremath{\widetilde{\Omega}_{\omega}{\pv{#1}}}}

\let\op=\llbracket
\let\cl=\rrbracket

\let\ge\geqslant

\title[Finite basis and finite rank for pseudovarieties of semigroups]
{A note on the finite basis and finite rank properties for pseudovarieties of semigroups}

\author{J. Almeida \and M. H. Shahzamanian}

\address{Centro de Matem\'atica e Departamento de Matem\'atica,
  Faculdade de Ci\^encias, Universidade do Porto, Rua do Campo Alegre,
  687, 4169-007 Porto, Portugal.
  \href{mailto:jalmeida@fc.up.pt}{jalmeida@fc.up.pt};
  \href{m.h.shahzamanian@fc.up.pt}{m.h.shahzamanian@fc.up.pt}}

\begin{document}

\begin{abstract}
  The finite basis property is often connected with the finite rank
  property, which it entails. Many examples have been produced of
  finite rank varieties which are not finitely based. In this note, we
  establish a result on nilpotent pseudovarieties which yields
  many similar examples in the realm of pseudovarieties of semigroups.
\end{abstract}



\maketitle

Recall that a (pseudo)variety is a class of (finite) algebras of a
given type which is closed under the operators of taking homomorphic
images, subalgebras and (finitary) direct products. By theorems of
Birkhoff~\cite{Birkhoff:1935} and Reiterman~\cite{Reiterman:1982}, we
may define a (pseudo)variety by a set of (pseudo)identities, which is
called a \emph{basis}. We say that a (pseudo)variety is finitely based
if it admits a finite basis of (pseudo)identities. The reader is
referred to~\cite{Almeida:1994a} for more details and basic
definitions on this topic.

An algebra $S$ is said to be \emph{$n$-generated} if there is some
subset of~$S$ with at most $n$~elements that generates $S$.
A pseudovariety \pv V is said to \emph{have rank at most
  $n$} if, whenever a finite algebra $S$ is such that all its
$n$-generated subalgebras belong to~\pv V, so does $S$. If there is
such an integer $n\ge0$, then \pv V is said to have \emph{finite rank}
and the minimum value of $n$ for which \pv V has rank at most $n$ is
called the \emph{rank} of~\pv V. If there is no such $n$, then \pv V
is said to have \emph{infinite rank}.

Similarly, a variety \Cl V of algebras is said to \emph{have rank at
  most $n$} if, whenever an algebra $S$ is such that all its
$n$-generated subalgebras belong to~\Cl V, so does $S$. The related
notions of finite and infinite rank are defined as for
pseudovarieties.

By \cite[Proposition IV.3.9]{Cohn}, we have the variety case of the
following proposition whose proof works as well for pseudovarieties.

\begin{Prop}
  \label{p:folklore}
  A (pseudo)variety has rank at most $n$ if and only if it admits a
  basis of (pseudo)identities in at most $n$ variables.
\end{Prop}

In view of Proposition~\ref{p:folklore}, the rank is also known as
\emph{axiomatic rank}. Proposition~\ref{p:folklore} says, in
particular, that if a (pseudo)variety is finitely based then it has
finite rank.
It is easy to see that the two properties are equivalent for a
(pseudo)variety generated by a single finite algebra: indeed, if the
rank is $n$ then the diagram of the relatively free algebra on $n$
free generators constitutes a finite basis of identities.
This is why most proofs that concrete finite algebras are infinitely
based, in the sense that they generate infinitely based
(pseudo)varieties, establish directly that the generated
(pseudo)varieties have infinite rank. A case which has deserved a lot
of attention is that of finite semigroups, where the main open
question is whether one can effectively decide whether a given finite
semigroup is finitely based~\cite{Volkov:2001}. Another case of
interest is that of varieties generated by a semigroup given by a
one-relator presentation, for which again the two properties turn out
to be equivalent~\cite{Shneerson:1989}.

The two properties, having finite rank and being finitely based, are
not, in general, equivalent. Examples of varieties of semigroups which
have finite rank and are infinitely based are well known
(see for example Tables 1.3 and 1.4 in \cite{She-Vol}). Many such
examples, as those in \cite[Table~1.3]{She-Vol}, are actually of
non-finitely based systems of balanced identities involving a (small)
finite number of variables.

Recall that $\pv N=\op x^\omega=0\cl$ is the pseudovariety of all
finite \emph{nil} (or \emph{power nilpotent}) semigroups. The purpose
of this note is to establish the theorem below which shows how to
obtain infinitely based pseudovarieties of nil semigroups from
infinitely based varieties of semigroups defined by balanced
identities.

A major difference between dealing with identities and
pseudoidentities is the absence of an analog for pseudoidentities of
the completeness theorem for equational
logic~\cite[Theorem~3.8.8]{Almeida:1994a}, in the sense that there is
no complete and sound finite deductive system for pseudoidentities.
Instead, we invoke the general theory developed
in~\cite[Section~3.8]{Almeida:1994a} for which an alternative has
recently been proposed in the form of a deductive system involving
infinite proofs \cite{Almeida&Klima:2017a}.

We start by recalling some notation and terminology. We fix a sequence
$(x_n)_{n\ge1}$ of distinct variables and we denote by \Om nN the
pro-\pv N semigroup freely generated by $\{x_1,\ldots,x_n\}$. For
$m<n$, we view \Om mN as being naturally embedded in~\Om nN via the
unique continuous homomorphism that sends each $x_i$ ($i=1,\ldots,m$)
to itself. This leads to a topological semigroup \tilom N which is the
inductive limit of the sequence $\Om1N\to\Om2N\to\Om3N\to\cdots$. Note
that, since each \Om nN is the one-point compactification of the
discrete free semigroup $\{x_1,\ldots,x_n\}^+$, obtained by adding a
zero, \tilom N is also obtained from the discrete free semigroup
$\{x_1,x_2,\ldots\}^+$ by adding a zero.

A subset $K$ of~\tilom N is clopen if and only if $K\cap\Om nN$ is
clopen for every $n\ge1$, that is, $K\cap\Om nN$ is a finite subset of
$\{x_1,\ldots,x_n\}^+$ or it is the complement in~\Om nN of such a
subset. In other words, $K$ is clopen in~\tilom N if and only if
$0\notin K$ and each intersection $K\cap\Om nN$ is finite, or $0\in K$
and each set $\{x_1,\ldots,x_n\}^+\setminus K$ is finite.

Note that the space \tilom N is not compact: for instance, the
sequence $(x_n)_n$ has no convergent subsequence. In fact, a sequence
of words $(u_n)_n$ converges in~\tilom N if and only if it is
eventually constant, or there is $N$ such that
$c(u_n)\subseteq\{x_1,\ldots,x_N\}$ for all~$n$ and $|u_n|\to\infty$,
in which case the limit is zero. Indeed, for any sequence $(x_n)_n$ in
\tilom N with only a finite number of terms in each \Om nN, its
complement is an open set in~\tilom N. As another example, the
sequence $(x_1\cdots x_n)_n$ has no convergent subsequence in~\tilom N
even though the length of its terms tends to infinity.
  
As in~\cite[Section~3.8]{Almeida:1994a}, we consider a subset
$\Lambda$ of $\tilom N\times\tilom N$, which is viewed as an arbitrary
set of pseudoidentities for nil semigroups. We say that $\Lambda$ is
\emph{strongly closed} if $\Lambda$ is a fully invariant congruence on
the semigroup~\tilom N and the clopen unions of $\Lambda$-classes
separate the classes of~$\Lambda$.
By~\cite[Corollary~3.8.4]{Almeida:1994a}, the strongly closed sets of
nil pseudoidentities are precisely the sets of all pseudoidentities
that are valid in a pseudovariety of nil semigroups.

\begin{Thm}\label{t:general}
  Let $\Sigma$ be a set of balanced semigroup identities such that the
  variety $[\Sigma]$ is infinitely based. Then, the pseudovariety $\pv
  N\cap\op\Sigma\cl$~is infinitely based.
\end{Thm}

\begin{proof}
  Suppose that $\pv W=\pv N\cap\op\Sigma\cl$ admits a finite basis
  $\Sigma_0$ of pseudoidentities. We may as well assume that the
  pseudoidentity $x^\omega=0$ belongs to~$\Sigma_0$. In the presence
  of that pseudoidentity, every pseudoidentity that is not its
  consequence is equivalent to either an identity $u=v$ or to a
  pseudoidentity of the form $u=0$, where $u$ is a word. The latter
  case is excluded for pseudoidentities in~$\Sigma_0$ because no such
  pseudoidentity is valid in~\pv W since \pv W contains all monogenic
  aperiodic semigroups. In the former case, if the identity $u=v$ is
  not balanced then, together with the pseudoidentity $x^\omega=0$, it
  entails a pseudoidentity of the form $w=0$, for some word $w$, which
  has already been excluded. Thus, we may as well assume that
  $\Sigma_0$ consists of the pseudoidentity $x^\omega=0$ together with
  a finite set $\Sigma_1$ of balanced identities.

  By the above cited results from~\cite[Section~3.8]{Almeida:1994a},
  an identity $u=v$ is valid in $\pv W=\pv N\cap\op \Sigma_1\cl$, if
  and only if it belongs to every strongly closed set of nil
  pseudoidentities containing~$\Sigma_1$. We claim that the smallest
  strongly closed set of nil pseudoidentities containing $\Sigma_1$ is
  the fully invariant congruence $\Lambda$ on $\{x_1,x_2,\ldots\}^+$
  generated by~$\Sigma_1$ together with the trivial pseudoidentity
  $0=0$. Indeed, since the set $\theta=\Lambda\cup\{0=0\}$ is
  certainly contained in every fully invariant congruence on~\tilom N
  containing~$\Sigma_1$, it suffices to show that $\theta$ is strongly
  closed. For this purpose, we start by noting that the
  $\theta$-classes are the $\Lambda$-classes, which are finite sets
  since they consist of words of equal length involving the same
  letters, together with the singleton set $\{0\}$. Since finite
  subsets of $\{x_1,x_2,\ldots\}^+$ are clopen in~\tilom N, to
  separate a class $C$ different from $\{0\}$ from any other class by
  a clopen union of classes, it suffices to take $C$ itself.

  We thus conclude that an identity $u=v$ is valid in the
  pseudovariety $\pv N\cap\op\Sigma_1\cl$ if and only if it is valid
  in the variety $[\Sigma_1]$. In particular, we obtain the equality
  $[\Sigma_1]=[\Sigma]$, which contradicts the assumption that the
  variety $[\Sigma]$ is infinitely based. Hence, \pv W is infinitely
  based.
\end{proof}

Note that the rank of $\pv N\cap\op\Sigma\cl$ is at most $\max\{2,n\}$,
where $n$ is the rank of the variety~$[\Sigma]$. Hence, combining
Theorem~\ref{t:general} with examples of infinitely based varieties of
finite rank, we obtain the following result, which is the main purpose
of this note.

\begin{Cor}
  \label{c:main}
  There are pseudovarieties of (nil) semigroups that have finite rank
  and are infinitely based.
\end{Cor}

For instance, since the varieties $[yx^ny=xyx^{n-2}yx\ (n\ge3)]$ and
$[xyx^ny=yx^nyx\ (n\ge2)]$ are infinitely based (see, respectively,
\cite{Per} and~\cite{Lyapin}), their intersections with \pv N are also
infinitely based but both have rank two.

Note that the examples of infinitely based finite rank pseudovarieties
resulting from Theorem~\ref{t:general} are not locally finite. We do
not know if there are any locally finite infinitely based finite rank
pseudovarieties or even varieties with the same properties.

\subsubsection*{Acknowledgments}
We greatly appreciate the comments of the Editor.
This work was supported, in part, by CMUP (UID/MAT/ 00144/2013), which
is funded by FCT (Portugal) with national (MCTES) and European
structural funds (FEDER), under the partnership agreement PT2020.
The work of the second author was also partly supported by the FCT
post-doctoral scholarship SFRH/BPD/89812/2012.

\bibliographystyle{amsplain}%
\bibliography{ibfr_20}

\end{document}